\documentclass[12pt,reqno]{amsart}
\usepackage{amssymb,amsmath,amsthm,enumerate,verbatim,bbm}

\usepackage[a4paper]{geometry}

\DeclareMathOperator{\diag}{diag}

\DeclareMathOperator{\supp}{supp}

\newcommand{\abs}[1]{\lvert#1\rvert}
\newcommand{\Abs}[1]{\left\lvert#1\right\rvert}
\newcommand{\norm}[1]{\lVert#1\rVert}

\newcommand{\bbT}{{\mathbb T}}
\newcommand{\bbR}{{\mathbb R}}
\newcommand{\bbC}{{\mathbb C}}

\newcommand{\bbN}{{\mathbb N}}
\newcommand{\bbZ}{{\mathbb Z}}

\newcommand{\bbD}{{\mathbb D}}

\newcommand{\calH}{{\mathcal H}}

\newcommand{\Sch}{\mathbf{S}}
\numberwithin{equation}{section}

\theoremstyle{plain}
\newtheorem{theorem}{\bf Theorem}[section]
\newtheorem*{theorem*}{Theorem 1.1$'$}
\newtheorem{lemma}[theorem]{\bf Lemma}

\newtheorem{corollary}[theorem]{\bf Corollary}

\theoremstyle{definition}

\theoremstyle{remark}
\newtheorem*{remark*}{\bf Remark}
\newtheorem{remark}[theorem]{\bf Remark}

\newcommand{\wt}{\widetilde}
\newcommand{\eps}{\varepsilon}

\newcommand{\1}{\mathbbm{1}}

\renewcommand{\[}{\begin{equation}}
\renewcommand{\]}{\end{equation}}

\newcommand{\f}{{\varphi}}

\begin{document}

\title[Spectral asymptotics for Toeplitz operators]{Spectral asymptotics for a class of Toeplitz operators on the Bergman space}

\author{Alexander Pushnitski}

%\date{1 October 2017}

\subjclass[2010]{47B32, 47B36}

\keywords{Toeplitz operators, Bergman space, banded matrix, spectral asymptotics}

\begin{abstract}
We consider a class of compact Toeplitz operators on the Bergman space on the unit disk. 
The symbols of operators in our class are assumed to have a sufficiently regular power-like behaviour
near the boundary of the disk. We compute the asymptotics of singular values of this class
of Toeplitz operators. We use this result to obtain the asymptotics of singular values for 
a class of banded matrices. 
\end{abstract}

\maketitle

%%%%%%%%%%%%%%%%%%%%%%%%%%%%%%%
%%%%%%%%%%%%%%%%%%%%%%%%%%%%%%%
\section{Introduction and main results}\label{sec.a}
%%%%%%%%%%%%%%%%%%%%%%%%%%%%%%%
%%%%%%%%%%%%%%%%%%%%%%%%%%%%%%%
%%%%%%%%%%%%%%%%%%%%%%%%%%%%%%%
\subsection{Introduction}
%%%%%%%%%%%%%%%%%%%%%%%%%%%%%%%
Let $\bbD$ be the unit disc in the complex plane, and let 
$L^2(\bbD)$ be the Hilbert space of all square integrable functions 
with respect to the normalised Lebesgue area measure on $\bbD$.  
Next, let $B^2(\bbD)$ be the Bergman space, i.e. the 
closure of the linear span of the functions $\{z^n\}_{n=0}^\infty$ 
in $L^2(\bbD)$. We denote by $P:L^2(\bbD)\to B^2(\bbD)$ the 
Bergman projection, i.e. the orthogonal projection in $L^2$ onto $B^2$. 
For a \emph{symbol} $\varphi\in L^\infty(\bbD)$, 
the Toeplitz operator $T(\varphi)$ in $B^2(\bbD)$
is defined by
$$
T(\varphi)f=P(\varphi\cdot f), \quad f\in B^2(\bbD). 
$$
It is well known that if $\varphi(z)\to0$ as $\abs{z}\to1$, then $T(\varphi)$ is compact. 
Moreover, roughly speaking, the rate of convergence of $\varphi(z)\to0$ as $\abs{z}\to1$
determines the rate of convergence of the sequence of singular values 
$s_n(T(\varphi))\to 0$ as $n\to\infty$.
For radial symbols $\varphi(z)=\varphi(\abs{z})$ it is very easy to make this statement
precise. Indeed, in this case the Toeplitz operator $T(\varphi)$ is diagonal in the standard
orthonormal basis $\{\sqrt{n+1}z^n\}_{n=0}^\infty$ of $B^2(\bbD)$, and so the sequence
of the singular values  of $T(\varphi)$ is given by 
$$
s_n(T(\varphi))
=
(n+1)\Abs{(T(\varphi)z^n,z^n)_{L^2(\bbD)}}
=
2(n+1)\Abs{\int_0^1 r^{2n}\varphi(r) rdr}. 
$$
Specialising to the power behaviour $\varphi(r)=(1-r)^\gamma$, $\gamma>0$, by 
an elementary calculation one obtains
\[
s_n(T(\varphi))
=
2^{-\gamma}\Gamma(\gamma+1)n^{-\gamma} +O(n^{-1-\gamma}), \quad n\to\infty.
\label{a1b}
\]
The purpose of this paper is 

(i) to present a method that allows one to extend this calculation to the symbols $\varphi$ 
that have some sufficiently regular angular dependence; 

(ii) to give an application to the spectral analysis of banded matrices.

\subsection{Spectral asymptotics for Toeplitz operators}

%%%%%%%%%%%%%%%
\begin{theorem}\label{thm.a1}
%%%%%%%%%%%%%%%
Let $\f\in L^\infty(\bbD)$, and let $\gamma>0$. 
Assume that for some continuous function $\f_\infty$ on the unit circle one has 
$$
\sup_{0<\theta\leq 2\pi}\abs{(1-r)^{-\gamma}\f(re^{i\theta})-\f_\infty(e^{i\theta})}\to0,
\quad r\to1.
$$ 
Then the singular values of $T(\f)$ satisfy
$$
s_n(T(\f))
=
C_\gamma(\f)n^{-\gamma}+o(n^{-\gamma}), \quad n\to\infty,
$$
where 
$$
C_\gamma(\f)=2^{-\gamma}\Gamma(\gamma+1)
\biggl(\int_0^{2\pi}\abs{\f_\infty(e^{i\theta})}^{1/\gamma}\frac{d\theta}{2\pi}\biggr)^\gamma.
$$
\end{theorem}

As will be clear from the proof, the requirement of the continuity of $\varphi_\infty$ can be considerably
relaxed. For example, any Riemann integrable $\varphi_\infty$ is admissible.

If $\f$ is real-valued, then it is evident that $T(\f)$ is self-adjoint.  
In this case, let us denote by $\{\lambda_n^+(T(\f))\}_{n=0}^\infty$ the 
sequence of positive eigenvalues of $T(\f)$. One can prove an analogous formula 
\[
\lambda_n^{+}(T(\f))
=
C_\gamma^+(\f)n^{-\gamma}+o(n^{-\gamma}), \quad n\to\infty,
\label{a3d}
\]
where 
$$
C_\gamma^+(\f)=2^{-\gamma}\Gamma(\gamma+1)
\biggl(\int_0^{2\pi}\f_\infty^+(e^{i\theta})^{1/\gamma}\frac{d\theta}{2\pi}\biggr)^\gamma
$$
and  $a^+=\max\{a,0\}$. 
(Of course, a similar result holds true for the negative eigenvalues.)
We omit the details of this argument, but will give some comments below 
--- see Remark~\ref{rmk.b4}.

By using conformal mapping, it is possible to extend this result to other 
domains. However, we prefer to focus on the simplest case, as our aim is  to 
emphasize the method rather than the result.

 A result similar to Theorem~\ref{thm.a1} is known in the context of Toeplitz operators on the Fock space. 
In this case, the symbol $\varphi(z)$  depends on $z\in\bbC$, and the 
rate of convergence of the singular values of the corresponding Toeplitz
operators to zero depends on the rate of convergence of $\varphi(z)\to0$ as $\abs{z}\to\infty$. 
In \cite{Raikov}, symbols that behave as $\abs{z}^{-\gamma}\varphi_\infty(e^{i\theta})$ when $\abs{z}\to\infty$
are considered and the spectral asymptotics of the type \eqref{a3d} is proved. 
The proof is achieved through a reduction to a pseudodifferential operator in $L^2(\bbR)$.

Another closely related recent result is \cite{BR}, where the authors consider 
harmonic Toeplitz operators in a bounded domain in $\bbR^d$ with symbols that have a power decay 
near the boundary. They obtain asymptotics of eigenvalues very similar to \eqref{a3d}. 
The method of proof of \cite{BR} is quite different from ours and consists in a reduction to 
a pseudodifferential operator on the boundary. The same method can probably 
be applied to give an alternative proof of Theorem~\ref{thm.a1}.

%%%%%%%%%%%%%%%%%%%%%%%%%%%%%%%%%%%
\subsection{Application: spectral asymptotics for banded matrices}
%%%%%%%%%%%%%%%%%%%%%%%%%%%%%%%%%%%

Let $A$ be an operator on $\ell^2$ corresponding to the 
infinite matrix $\{a_{j,k}\}_{j,k=0}^\infty$ of the following form. 
Our first assumption is that the matrix $A$ is \emph{banded}, i.e. 
for some $M\in\bbN$, we have
$$
a_{j,j+m}=0 \quad \text{ if $\abs{m}>M$.}
$$
Our second assumption is that for each $m$ with $\abs{m}\leq M$, the 
sequence of entries $a_{j,j+m}$ has a power asymptotics as $j\to\infty$. 
More precisely, we fix an exponent $\gamma>0$ and 
complex numbers $b_m$, $m=-M,\dots,M$, and assume that
$$
a_{j,j+m}=b_m j^{-\gamma}+o(j^{-\gamma}), \quad j\to\infty, \quad \abs{m}\leq M.
$$
It is easy to see that under these assumptions the matrix $A$ is compact. 
It is also not difficult to see that $s_n(A)=O(n^{-\gamma})$. 
The theorem below gives the asymptotic behaviour of the singular values. 
%%%%%%%%%%%%%%%%%
\begin{theorem}\label{thm.a3}
%%%%%%%%%%%%%%%%%
Under the above assumptions, the singular values of $A$ satisfy
$$
s_n(A)=
\biggl(\int_{-\pi}^\pi \abs{b(e^{i\theta})}^{1/\gamma}\frac{d\theta}{2\pi}\biggr)^{\gamma}n^{-\gamma}+o(n^{-\gamma}), 
\quad
n\to\infty,
$$
where 
\[
b(e^{i\theta})=\sum_{k=-M}^{M}b_k e^{ik\theta}, \quad 0<\theta\leq 2\pi.
\label{e4}
\]
\end{theorem}

\begin{remark*}
If $A$ is self-adjoint, then $b$ is necessarily real-valued. 
In this case one can prove a similar asymptotic formula for 
positive eigenvalues of $A$,  with $(b^+)^{1/\gamma}$ instead of $\abs{b}^{1/\gamma}$.  
\end{remark*}

%%%%%%%%%%%%%%%%%%%%%%%%%%%%%%%%%%%%%
\subsection{Key ideas of the proof and the structure of the paper}
%%%%%%%%%%%%%%%%%%%%%%%%%%%%%%%%%%%%%

The main issue is to prove Theorem~\ref{thm.a1} for symbols of the form 
$$
\varphi(re^{i\theta})=(1-r)^\gamma \varphi_\infty(e^{i\theta}). 
$$
For such symbols, we shall write $T_\gamma(\varphi_\infty)$ instead of $T(\varphi)$. 
By a limiting argument, the problem reduces to the case of functions $\varphi_\infty$
that are constant on each arc 
$$
\delta_\ell=\{e^{i\theta}: \tfrac{2\pi \ell}{L}<\theta<\tfrac{2\pi (\ell+1)}{L}\},
$$
where $L\in\bbN$ is fixed. Let us denote by 
$\1_{\delta_\ell}$ the characteristic function of $\delta_\ell$ and write
\[
\varphi_\infty=\sum_{\ell=1}^{L} c_\ell \1_{\delta_\ell}
\label{a6}
\]
with some coefficients $c_\ell$. Then 
\[
T_\gamma(\varphi_\infty)=\sum_{\ell=1}^{L} c_\ell T_\gamma(\1_{\delta_\ell}). 
\label{a5}
\]
Our main observation is that the terms in the sum \eqref{a5} are 
\emph{asymptotically orthogonal} in the sense that the products
$T_\gamma(\1_{\delta_\ell})T_\gamma(\1_{\delta_m})$, $\ell\not=m$, satisfy certain
Schatten class properties. 
An operator theoretic lemma from \cite{BS2} (Theorem~\ref{thm.b2}) then shows that the leading
term coefficient in spectral asymptotics is \emph{additive} for the sum \eqref{a5}. 
This means that the leading term asymptotics of singular values of $T_\gamma(\varphi_\infty)$ is 
the same as that of the orthogonal sum
$$
\oplus_{\ell=1}^\infty c_\ell T_\gamma(\1_{\delta_\ell}). 
$$
Using this fact, it is not difficult to derive the required spectral asymptotics for piecewise
constant functions $\varphi_\infty$, as in \eqref{a6}, from the same asymptotics for constant 
functions $\varphi_\infty$.

We note that this construction  is almost purely operator theoretic and 
uses very little specific features of the problem, apart from rotational symmetry.
Thus, it can be used in other contexts, for example for 
Toeplitz operators associated with multi-dimensional domains with spherical symmetry. 
This method of proof, based on asymptotic orthogonality, was developed in 
\cite{PuYa,PuYa2} for a different purpose.

Let us explain the key idea of the proof of Theorem~\ref{thm.a3}.
Consider the symbol of the form
$$
\varphi(re^{i\theta})=(1-r)^\gamma e^{im\theta},
$$
where $m\in\bbZ$. 
Then the matrix of $T(\varphi)$ in the standard orthonormal basis $\{\sqrt{k+1}z^k\}_{k=0}^\infty$ 
of the Bergman space is
$$
a_{j,k}=\sqrt{j+1}\sqrt{k+1}(\varphi z^j,z^k)_{L^2(\bbD)}.
$$
It is easy to see that $a_{j,k}\not=0$ only if $k=j+m$ and 
$$
a_{j,j+m}=2^{-\gamma}\Gamma(\gamma+1)j^{-\gamma}+o(j^{-\gamma})
$$
as $j\to\infty$. This corresponds to a matrix $A$ as in Theorem~\ref{thm.a3} with
only one non-zero diagonal. Taking a finite linear combination of such matrices gives the case
of a general banded matrix. Thus, the proof of Theorem~\ref{thm.a3} reduces to 
Theorem~\ref{thm.a1}. 

In Section~\ref{sec.b} we recall some background facts concerning singular value
asymptotics for compact operators and state the result from \cite{BS2} on 
asymptotically orthogonal operators. 
In Section~\ref{sec.c} we prove Theorem~\ref{thm.a1} and in Section~\ref{sec.e} 
we prove Theorem~\ref{thm.a3}.

\subsection{Acknowledgements}
The author is grateful to  P.~Honor\'e, G.~Raikov and M.~Reguera for useful discussions.

%%%%%%%%%%%%%%%%%%%%%%%%%%%%%%%
%%%%%%%%%%%%%%%%%%%%%%%%%%%%%%%
\section{Operator theoretic tools}\label{sec.b}
%%%%%%%%%%%%%%%%%%%%%%%%%%%%%%%
%%%%%%%%%%%%%%%%%%%%%%%%%%%%%%%

Here we collect some general operator theoretic tools related to the 
singular value estimates and asymptotics for compact operators. 

%%%%%%%%%%%%%%%%%%%%%%%%%%%%%%%%
\subsection{Definitions}
%%%%%%%%%%%%%%%%%%%%%%%%%%%%%%%%

For a compact operator $T$ in a Hilbert space, 
we denote by $\{s_n(T)\}_{n=0}^\infty$ the non-increasing
sequence of singular values of $T$, enumerated with multiplicities taken into account. 
Recall that $s_n(T)$ is defined as the $n$'th eigenvalue of the positive semi-definite 
self-adjoint operator $\abs{T}=\sqrt{T^*T}$. 
It will be convenient to work with the singular value counting function:
$$
n(s;T)=\#\{n: s_n(T)>s\}, \quad s>0.
$$
For $p>0$, the standard Schatten class $\Sch_p$ is defined by the condition
$$
\sum_{n=0}^\infty s_n(T)^p<\infty.
$$
In terms of the counting function $n(s;T)$ this can be equivalently rewritten as 
$$
\int_0^\infty n(s;T)s^{p-1}ds<\infty. 
$$
The weak Schatten class $\Sch_{p,\infty}$  is defined by the condition
$$
n(s;T)=O(s^{-p}), \quad s\to0.
$$
The quantity
$$
\norm{T}_{\Sch_{p,\infty}}=\sup_{s>0} sn(s;T)^{1/p} 
$$
is a quasi-norm on $\Sch_{p,\infty}$, and we will be considering $\Sch_{p,\infty}$ with 
respect to the topology given by this quasi-norm.
The subclass $\Sch_{p,\infty}^0\subset\Sch_{p,\infty}$ is defined by the condition
$$
n(s;T)=o(s^{-p}), \quad s\to0;
$$
it can be characterised as the closure of all operators of finite rank in $\Sch_{p,\infty}$. 

We recall that for a compact operator $T$, the relations
\[
s_n(T)=\frac{C}{n^\gamma}+o(n^{-\gamma}), \quad
n\to\infty, 
\quad \text{ and } \quad 
n(s;T)=\frac{C^{1/\gamma}}{s^{1/\gamma}}+o(s^{-1/\gamma}), \quad s\to0,
\label{a2a}
\]
are equivalent. In order to work with the asymptotic coefficient $C$ in such relations,
it will be convenient to use the following functionals in $\Sch_{p,\infty}$: 
\[
\Delta_p(T):=\limsup_{s\to0} s^p n(s;T), 
\quad
\delta_p(T):=\liminf_{s\to0} s^p n(s;T). 
\label{a2b}
\]
In applications, one usually has $\Delta_p(T)=\delta_p(T)$, but it is technically 
convenient to consider the upper and lower limits separately. 
The functionals $\Delta_p$ and $\delta_p$ are continuous in $\Sch_{p,\infty}$.

We will denote $\Sch_0=\cap_{p>0}\Sch_p$. In other words, $\Sch_0$ consists
of compact operators $T$ such that for all $C>0$, one has
$$
s_n(T)=O(n^{-C}), \quad n\to\infty. 
$$

%%%%%%%%%%%%%%%%%%%%%%%%%%%%%%%%
\subsection{Additive and multiplicative estimates}
%%%%%%%%%%%%%%%%%%%%%%%%%%%%%%%%
Below we recall some estimates for singular values of sums and products of compact operators.

The following fundamental result is known as the Ky Fan lemma; see e.g. \cite{BS}. 

%%%%%%%%%%%%%%%%%%%%%%%
\begin{lemma}\label{lma.b1}
%%%%%%%%%%%%%%%%%%%%%%%
Let $A\in\Sch_{\infty}$ and $B\in\Sch_{p,\infty}^0$ for some $p>0$. 
Then 
$$
\Delta_p(A+B)=\Delta_p(A), \quad
\delta_p(A+B)=\delta_p(A).
$$
\end{lemma}

In Section~\ref{sec.c} we will also need more advanced information about the quantities 
$\Delta_p$ and $\delta_p$. 
One has the following additive esimates, see e.g. \cite[formulas (11.6.12), (11.6.14), (11.6.15)]{BS}:
\begin{align}
\Delta_p(A_1+A_2)^{1/(p+1)}
&\leq 
\Delta_p(A_1)^{1/(p+1)}
+
\Delta_p(A_2)^{1/(p+1)},
\label{b1}
\\
\abs{\Delta_p(A_1)^{1/(p+1)}-\Delta_p(A_2)^{1/(p+1)}}
&\leq 
(\Delta_p(A_1-A_2))^{1/(p+1)},
\label{b2}
\\
\abs{\delta_p(A_1)^{1/(p+1)}-\delta_p(A_2)^{1/(p+1)}}
&\leq 
(\Delta_p(A_1-A_2))^{1/(p+1)}.
\label{b3}
\end{align}

We will also need some multiplicative estimates. 
One has (see \cite[(11.1.19), (11.1.12)]{BS})
\begin{gather}
n(s_1s_2;A_1A_2)\leq n(s_1;A_1)+n(s_2;A_2),\quad s_1>0, \quad s_2>0,
\label{b6}
\\
n(s;A_1A_2)\leq n(s;\norm{A_1}A_2), \quad s>0.
\label{b7}
\end{gather}
From \eqref{b6} it is not difficult to obtain (see \cite[(11.6.18)]{BS}) the bound
\[
\Delta_{p/2}(A_1A_2)\leq 2\Delta_p(A_1)\Delta_p(A_2).
\label{b4}
\]

%%%%%%%%%%%%%%%%%%%%%%%%%%%%%
\subsection{Asymptotically orthogonal operators}
%%%%%%%%%%%%%%%%%%%%%%%%%%%%%

The theorem below is the key operator theoretic ingredient 
of our construction. It has first appeared (under slightly more restrictive assumptions) 
in \cite[Theorem 3]{BS2}. 
Here we follow the presentation of  \cite{PuYa}. 

%%%%%%%%%%%%%%%%
\begin{theorem}\label{thm.b2}\cite{BS2,PuYa}
%%%%%%%%%%%%%%%%
Let $p>0$.  
Assume that $A_1,\dots,A_L\in\Sch_{\infty}$ and
\begin{equation}
A_\ell^*A_j\in\Sch_{p/2,\infty}^0,
\quad 
A_\ell A_j^*\in\Sch_{p/2,\infty}^0
\quad 
\text{ for all $\ell\not=j$.}
\label{b9}
\end{equation}
Then for   $A= A_1+\dots+A_L$, we have
\begin{align*}
\Delta_p(A)&=\limsup_{s\to0} s^{p}\sum_{\ell=1}^L n(s,A_\ell),
\\
\delta_p(A)&=\liminf_{s\to0} s^{p}\sum_{\ell=1}^L n(s,A_\ell).
\end{align*}
\end{theorem}

\begin{proof}
Put
$$
\calH^L=\underbrace{\calH\oplus\dots\oplus\calH}_{\text{$L$ terms}}
$$
and let
$A_0=\diag\{A_1,\dots,A_L\}$ in $\calH^L$, i.e., 
$$
A_0(f_1,\dots,f_L)=(A_1f_1,\dots,A_Lf_L).
$$
Since
$$
A_0^*A_0=\diag\{A_1^*A_1,\dots,A_L^*A_L\},
$$
we see that
$$
n(\eps;A_0)
=
\sum_{\ell =1}^L
n(\eps;A_\ell).
$$
Thus, we need to prove the relations
$$
\Delta_p(A)=\Delta_p(A_0), \quad 
\delta_p(A)=\delta_p(A_0).
$$
We will focus on the functionals $\Delta_p$; the functionals $\delta_p$
are considered in the same way.

Next, let 
$J:\calH^L\to\calH$ be the operator given by 
$$
J(f_1,\dots,f_L)=f_1+\dots+f_L
\quad \mbox{so that} \quad 
J^* f=(f,\dots,f).
$$
Then
$$
J A_0(f_1,\dots,f_L)=A_1f_1 + \dots +A_Lf_L
$$
and
$$
(J A_0)^* f=(A_1^*f,\dots,A_L^*f).
$$
It follows that
\begin{equation}
(J A_{0}) (J A_0)^* f= (A_1 A_1^* +\dots +A_L A_L^*) f
\label{b13}
\end{equation}
and the operator $(J A_0)^*(J A_0)$ is a ``matrix'' in $\calH^L$ given by
$$
(J A_0)^*(J A_0)
=
\begin{pmatrix}
A_1^*A_1 & A_1^*A_2 & \dots & A_1^* A_L
\\
A_2^*A_1 & A_2^*A_2 & \dots & A_2^* A_L
\\
\vdots & \vdots & \ddots & \vdots
\\
A_L^*A_1 & A_L^*A_2 & \dots  & A_L^* A_L
\end{pmatrix}.
$$
By our assumption \eqref{b9},   we have
\[
(J A_0)^*(J A_0)-A_0^*A_0\in\Sch_{p/2,\infty}^0.
\label{b15}
\]
Indeed,   the ``matrix'' of the operator in \eqref{b15}  has zeros on the 
diagonal, and its
off-diagonal elements are given by $A_\ell^* A_j$, $\ell\not=j$. 
Now Lemma~\ref{lma.b1} implies that
$$
\Delta_{p/2}((J A_0)^*(J A_0))
=
\Delta_{p/2}(A_0^*A_0)
$$
or
\[
\Delta_{p/2}((J A_0)(J A_0)^*)
=
\Delta_{p/2}(A_0^*A_0)
\label{b16}
\]
because for any compact operator $T$ the non-zero singular values of $T^*T$ and $TT^*$ coincide.

Further, since
$
AA^*=
\sum_{\ell,j=1}^L A_\ell A_j^*,
$
it follows from \eqref{b13} and the second assumption  \eqref{b9} that
$$
AA^*-(J A_0)(J A_0)^* = \sum_{j\not=\ell}A_\ell A_j^* \in\Sch_{p/2,\infty}^0.
$$
Using Lemma~\ref{lma.b1} again, from here we obtain
$$
\Delta_{p}(A)=\Delta_{p/2}(AA^*)=\Delta_{p/2}((J A_0)(J A_0)^*).
$$
Combining the last equality with    \eqref{b16}, we see that
$$
\Delta_p(A)=\Delta_{p/2}(A_0^*A_0)=\Delta_p(A_0).
$$
The same reasoning also proves $\delta_p(A)=\delta_p(A_0)$. 
\end{proof}

%%%%%%%%%%%%%%%%%%
\begin{corollary}\label{cr.b3}
%%%%%%%%%%%%%%%%%%
Under the hypothesis of the theorem above, assume in addition that 
$$
n(s;A_1)=n(s;A_2)=\dots=n(s;A_L), \quad s>0.
$$
Then 
$$
\Delta_p(A)=L \Delta_p(A_1),
\quad
\delta_p(A)=L \delta_p(A_1).
$$
\end{corollary}

\begin{remark}\label{rmk.b4}
In order to prove the asymptotics \eqref{a3d} for positive eigenvalues of self-adjoint operators, one needs
an analogue of Theorem~\ref{thm.b2} for the counting function of positive eigenvalues of self-adjoint operators. 
Such theorem can be found in \cite[Theorem~2.3]{PuYa2}. 
\end{remark}

%%%%%%%%%%%%%%%%%%%%%%%%%%%%%%%
%%%%%%%%%%%%%%%%%%%%%%%%%%%%%%%
\section{Proof of Theorem~\ref{thm.a1}}\label{sec.c}
%%%%%%%%%%%%%%%%%%%%%%%%%%%%%%%
%%%%%%%%%%%%%%%%%%%%%%%%%%%%%%%

%%%%%%%%%%%%%%%%%%%%%%%%%%%%%%%%%%%%%
\subsection{Preliminary remarks}
%%%%%%%%%%%%%%%%%%%%%%%%%%%%%%%%%%%%%

By the equivalence \eqref{a2a}, the statement of 
Theorem~\ref{thm.a1} can be equivalently rewritten in terms of the singular value counting function as
$$
\lim_{s\to0}s^{1/\gamma}n(s;T(\varphi))
=
\frac12\Gamma(\gamma+1)^{1/\gamma}
\int_0^{2\pi}\abs{\varphi_\infty(e^{i\theta})}^{1/\gamma}\frac{d\theta}{2\pi}.
$$
Throughout the proof, we use the shorthand notation for the coefficient appearing 
in the right side here:
$$
\varkappa_\gamma:=\frac12\Gamma(\gamma+1)^{1/\gamma}.
$$
Using this notation and the functionals $\Delta_p$, $\delta_p$ defined in \eqref{a2b}, 
one can rewrite the statement of 
Theorem~\ref{thm.a1} as
$$
\Delta_{1/\gamma}(T(\f))=\delta_{1/\gamma}(T(\f))=\varkappa_\gamma \int_0^{2\pi}\abs{\varphi_\infty(e^{i\theta})}^{1/\gamma}\frac{d\theta}{2\pi}.
$$
As in Section~\ref{sec.a}, for a symbol  $\f$ of the form
$$
\f(re^{i\theta})=(1-r)^\gamma g(e^{i\theta}), \quad g\in L^\infty(\bbT),
$$ 
we will write $T_\gamma(g)$ 
instead of $T(\varphi)$.
The case of the radially symmetric $\f$ corresponds to the choice $g=1$.
In this case, the asymptotics of singular values is given by \eqref{a1b}.
In terms of the asymptotic functionals $\Delta_p$, $\delta_p$ this can be rewritten as  
$$
\Delta_{1/\gamma}(T_\gamma(1))=\delta_{1/\gamma}(T_\gamma(1))
=
\varkappa_\gamma.
$$
Finally, we need some notation: for a symbol $\f$ we denote by $M(\varphi)$ the operator of
multiplication by $\varphi(z)$ in $L^2(\bbD)$. Then the Toeplitz operator $T(\f)$ can be written as
\[
T(\f)=PM(\f)P^* \quad \text{ in $B^2(\bbD)$,}
\label{c4a}
\]
where the orthogonal projection $P$ is understood to act from $L^2(\bbD)$ to $B^2(\bbD)$, 
and $P^*$ acts from $B^2(\bbD)$ to $L^2(\bbD)$.

%%%%%%%%%%%%%%%%%%%%%%%%%%%%%%%%%%%%%
\subsection{Asymptotic orthogonality}
%%%%%%%%%%%%%%%%%%%%%%%%%%%%%%%%%%%%%

The main analytic ingredient of our construction is the following lemma. 

%%%%%%%%%%%%%%%%
\begin{lemma}\label{lma.c2}
%%%%%%%%%%%%%%%%
Let $g_1,g_2\in L^\infty(\bbT)$ be such that the distance between the supports 
of $g_1$ and $g_2$ on $\bbT$ is positive. 
Then 
$$
T_\gamma(g_1)T_\gamma(g_2)^*\in\Sch_0. 
$$
\end{lemma}
\begin{proof}
For $j=1,2$, denote
\begin{align*}
\f_j(re^{i\theta})&=(1-r)^\gamma g_j(e^{i\theta}),
\\
\psi_j(re^{i\theta})&=(1-r)^\gamma \1_{[1/2,1]}(r)g_j(e^{i\theta}),
\end{align*}
where $ \1_{[1/2,1]}$ is the characteristic function of the interval $[1/2,1]$. 
Since $\f_j-\psi_j$ is bounded and supported in the disc $\abs{z}\leq1/2$, 
it is easy to conclude that 
$$
T(\f_j-\psi_j)\in\Sch_0, \quad j=1,2.
$$
Thus, it suffices to prove the inclusion
$$
T(\psi_1)T(\psi_2)^*\in\Sch_0. 
$$
We have 
$$
T(\psi_1)T(\psi_2)^*=PM(\psi_1)P^*PM(\overline{\psi_2})P^*,
$$
and so it suffices to prove the inclusion 
$$
M(\psi_1) P^*P M(\overline{\psi_2})\in\Sch_0.
$$ 
Further, 
let $\omega_1,\omega_2\in C^\infty(\overline{\bbD})$ be 
such that the distance between the supports 
$\supp\omega_1$, $\supp\omega_2$ is  positive  and 
$$
\omega_1\psi_1=\psi_1, \quad \omega_2\psi_2=\psi_2.
$$
Such functions exist by our assumption on the supports of $g_1$, $g_2$. 
We have
$$
M(\psi_1) P^*P M(\overline{\psi_2})
=
M(\psi_1)M(\omega_1)P^*PM(\overline{\omega_2})M(\overline{\psi_2}). 
$$
So it suffices to prove that 
$M(\omega_1) P^*P M(\overline{\omega_2})\in\Sch_0$.
Clearly, $P^*P$ is the orthogonal projection in $L^2(\bbD)$ whose integral kernel 
is the Bergman kernel. Using the explicit formula for the Bergman kernel, we see
that $M(\omega_1) P^*P M(\overline{\omega_2})$ 
is the integral operator in $L^2(\bbD)$ with the kernel
$$
\frac{ \omega_1(z)\overline{\omega_2(\zeta)}}{(1-z\overline{\zeta})^{2}},
\quad z,\zeta\in\bbD.
$$
Since $\omega_1$ and $\omega_2$ have disjoint supports, 
we see that this kernel
is $C^\infty$-smooth. 
It is a well known fact that integral operators with $C^\infty$ kernels on compact domains belong to $\Sch_0$ 
(it can be proven, for example, by approximating the integral kernel by
polynomials). 
Thus, the operator $M(\omega_1) P^*P M(\overline{\omega_2})$ 
is in $\Sch_0$.
\end{proof}

We would like to have an analogous statement 
where $g_1$ and $g_2$ are characteristic functions of disjoint (but possibly ``touching")
open intervals. We will obtain it from Lemma~\ref{lma.c2} by an approximation argument. 
To this end, in the next subsection we develop some rather crude estimates.

%%%%%%%%%%%%%%%%%%%%%%%%%%%%
\subsection{Auxiliary estimates}
%%%%%%%%%%%%%%%%%%%%%%%%%%%%

%%%%%%%%%%%%%%%%
\begin{lemma}\label{lma.c1a}
%%%%%%%%%%%%%%%%
If $\abs{g}\leq g_0$, where $g_0$ is a constant, then 
$$
\Delta_{1/\gamma}(T_\gamma(g))
\leq 
2\varkappa_\gamma \abs{g_0}^{1/\gamma}. 
$$
\end{lemma}
\begin{proof}
Let us write our symbol $\f$ as 
$$
\f=\f_0^{1/2}\f_1\f_0^{1/2}, \quad\text{ where }\quad \f_0(z)=\abs{g_0}(1-\abs{z})^{\gamma}
\quad\text{and}
\quad
\abs{\f_1(z)}\leq 1.
$$
Then by \eqref{c4a} we have
$$
n(s;T_\gamma(g))=n(s;T(\f))=n(s;GM(\f_1)G^*), \quad G=PM(\f_0^{1/2})
$$
and 
$$
n(s;\abs{g_0}T_\gamma(1))=n(s;T(\f_0))=n(s;PM(\f_0) P^*)=n(s;GG^*).
$$
Applying the estimates \eqref{b6}, \eqref{b7}, we obtain
\begin{multline*}
n(s;GM(\f_1)G^*)
\leq
n(\sqrt{s};G)+n(\sqrt{s};M(\f_1)G^*)
\\
\leq
n(\sqrt{s};G)+n(\sqrt{s};G^*)
=
2n(\sqrt{s};G)
=
2n(s;GG^*)=2n(s;T(\f_0)).
\end{multline*}
Multiplying by $s^{1/\gamma}$ and taking $\limsup$ yields 
$$
\Delta_{1/\gamma}(T_\gamma(g))
\leq 
2\Delta_{1/\gamma}(\abs{g_0}T_\gamma(1))
=
2\abs{g_0}^{1/\gamma}
\varkappa_\gamma,
$$
as required.
\end{proof}

%%%%%%%%%%%%%%%%
\begin{lemma}\label{lma.c3}
%%%%%%%%%%%%%%%%
Let $\delta\subset\bbT$ be an arc with 
arclength $\abs{\delta}<2\pi$. Then 
$$
\Delta_{1/\gamma}(T_\gamma(\1_\delta))
\leq 
\varkappa_\gamma\abs{\delta}.
$$
\end{lemma}
\begin{proof}
Let $L\in\bbN$ be such that $2\pi/(L+1)\leq \abs{\delta}< 2\pi/L$. 
For $\ell=1,\dots,L$, let $\delta_\ell$ be the arc $\delta$ rotated by the angle $2\pi \ell/L$: 
\begin{equation}
\delta_\ell
=
e^{2\pi i\ell/L}\delta_\ell. 
\label{c1}
\end{equation}
In particular, $\delta_L=\delta$. Then the arcs $\delta_1,\dots\delta_L$ are disjoint and so 
$$
g:=\sum_{\ell=1}^L \1_{\delta_\ell}\leq 1.
$$
By Lemma~\ref{lma.c1a}, it follows that 
$$
\Delta_{1/\gamma}(T_\gamma(g))\leq 2\varkappa_\gamma.
$$
Further, it is easy to see that the operators
$T_\gamma(\1_{\delta_\ell})$ are unitarily equivalent to each other by rotation. 
Thus, 
$$
n(s;T_\gamma(\1_{\delta_\ell}))=n(s;T_\gamma(\1_{\delta})), \quad s>0
$$
for all $\ell$. 
Finally, we have
$$
T_\gamma(\1_{\delta_\ell})T_\gamma(\1_{\delta_j})\in\Sch_0,\quad \ell\not=j,
$$
by Lemma~\ref{lma.c2}. 
Thus, we can apply Corollary~\ref{cr.b3} to $A_\ell=T_\gamma(\1_{\delta_\ell})$ 
and $A=T_\gamma(g)$. 
This yields
$$
\Delta_{1/\gamma}(T_\gamma(\1_\delta))
=
\Delta_{1/\gamma}(T_\gamma(g))/L
\leq
2\varkappa_\gamma/L
\leq
2\pi\varkappa_\gamma/(L+1)
\leq
\varkappa_\gamma\abs{\delta},
$$
as claimed.
\end{proof}

%%%%%%%%%%%%%%%%
\begin{lemma}\label{lma.c3a}
%%%%%%%%%%%%%%%%
Let $\delta$ and $\delta'$ be two arcs in $\bbT$ such that 
the symmetric difference $\delta\Delta\delta'$ has total length $<\eps$. 
Then 
$$
\Delta_{1/\gamma}(T_\gamma(\1_\delta))-T_\gamma(\1_{\delta'}))
\leq 
2^{1+1/\gamma}\varkappa_\gamma\eps.
$$
\end{lemma}
\begin{proof}
Let $\delta\Delta\delta'=\delta_1\cup\delta_2$, where $\delta_1$, $\delta_2$ are
intervals with $\abs{\delta_1}<\eps$, $\abs{\delta_2}<\eps$. 
Then 
$$
T_\gamma(\1_\delta)-T_\gamma(\1_{\delta'})
=
\pm T_\gamma(\1_{\delta_1})\pm T_\gamma(\1_{\delta_2}),
$$
where the signs depend on the relative location of $\delta$, $\delta'$. 
Using the estimate \eqref{b1}, we get
$$
\Delta_{1/\gamma}(T_\gamma(\1_\delta)-T_\gamma(\1_{\delta'}))^{\gamma/(\gamma+1)}
\leq 
\Delta_{1/\gamma}(T_\gamma(\1_{\delta_1}))^{\gamma/(\gamma+1)}
+
\Delta_{1/\gamma}(T_\gamma(\1_{\delta_2}))^{\gamma/(\gamma+1)},
$$
and so, applying Lemma~\ref{lma.c3}, we get 
$$
\Delta_{1/\gamma}(T_\gamma(\1_\delta)-T_\gamma(\1_{\delta'}))
\leq
\varkappa_\gamma(\abs{\delta_1}^{\gamma/(\gamma+1)}+\abs{\delta_2}^{\gamma/(\gamma+1)})^{1+1/\gamma}
\leq
\eps \varkappa_\gamma 2^{1+1/\gamma},
$$
as required. 
\end{proof}

Now we can prove a refined version of Lemma~\ref{lma.c2}, where the supports of $g_1$, $g_2$ 
are allowed to ``touch''.

%%%%%%%%%%%%%%%%
\begin{lemma}\label{lma.c4}
%%%%%%%%%%%%%%%%
Let $\delta$ and $\delta'$ be disjoint open arcs in $\bbT$: $\delta\cap\delta'=\varnothing$.
Then 
$$
T_\gamma(\1_\delta)T_\gamma(\1_{\delta'})\in\Sch_{1/2\gamma,\infty}^0.
$$
\end{lemma}
\begin{proof}
Let us ``shrink" $\delta$ a little: 
for $\eps>0$, let $\delta_\eps$ be an arc such that the distance between $\delta_\eps$
and $\delta'$ is positive and the symmetric difference $\delta_\eps\Delta\delta$
has a total length $<\eps$. By Lemma~\ref{lma.c2}, we have
$$
T_\gamma(\1_{\delta_\eps})T_\gamma(\1_{\delta'})\in\Sch_0\subset\Sch_{1/2\gamma,\infty}^0.
$$
By Lemma~\ref{lma.b1}, it follows that
$$
\Delta_{1/2\gamma}(T_\gamma(\1_\delta)T_\gamma(\1_{\delta'}))
=
\Delta_{1/2\gamma}\bigl((T_\gamma(\1_\delta)-T_\gamma(\1_{\delta_\eps}))T_\gamma(\1_{\delta'})\bigr).
$$
Applying the estimate \eqref{b4}, we get 
$$
\Delta_{1/2\gamma}\bigl((T_\gamma(\1_\delta)-T_\gamma(\1_{\delta_\eps}))T_\gamma(\1_{\delta'})\bigr)
\leq
\Delta_{1/\gamma}\bigl((T_\gamma(\1_\delta)-T_\gamma(\1_{\delta_\eps})\bigr)
\Delta_{1/\gamma}(T_\gamma(\1_{\delta'})).
$$
By Lemma~\ref{lma.c3a}, we get
$$
\Delta_{1/2\gamma}(T_\gamma(\1_\delta)T_\gamma(\1_{\delta'}))\leq C_\gamma\eps.
$$
Since $\eps$ can be chosen arbitrary small, we get
$$
\Delta_{1/2\gamma}(T_\gamma(\1_\delta)T_\gamma(\1_{\delta'}))=0,
$$
which is exactly what is required. 
\end{proof}

%%%%%%%%%%%%%%%%%%%%%%%%%%%%%%%%%%
\subsection{Step functions $g$}
%%%%%%%%%%%%%%%%%%%%%%%%%%%%%%%%%%
%%%%%%%%%%%%%%%%
\begin{lemma}\label{lma.c5}
%%%%%%%%%%%%%%%%
Let $\delta$ be an arc with $\abs{\delta}=2\pi/L$, $L\in\bbN$. 
Then 
$$
\Delta_{1/\gamma}(T_\gamma(\1_\delta))=
\delta_{1/\gamma}(T_\gamma(\1_\delta))=
\varkappa_\gamma/L. 
$$
\end{lemma}
\begin{proof}
Let $\delta_\ell$ be as in \eqref{c1}. 
Then $\delta_j\cap \delta_\ell=\varnothing$ for $j\not=\ell$ and 
$$
1=\sum_{\ell=1}^L \1_{\delta_\ell}\quad \text{ a.e. on $\bbT$.}
$$
Thus, 
$$
T_\gamma(1)=\sum_{\ell=1}^L T_\gamma(\1_{\delta_\ell})
$$
and by Lemma~\ref{lma.c4}
$$
T_\gamma(\1_{\delta_\ell})T_\gamma(\1_{\delta_j})\in\Sch_{1/2\gamma,\infty}^0.
$$
Thus, we can apply Corollary~\ref{cr.b3}, which yields
$$
\Delta_{1/\gamma}(T_\gamma(\1_\delta))
=
\Delta_{1/\gamma}(T_\gamma(1))/L,
$$
and similarly for the lower limits $\delta_{1/\gamma}$. 
\end{proof}

%%%%%%%%%%%%%%%%
\begin{lemma}\label{lma.c6}
%%%%%%%%%%%%%%%%
Let $\delta\subset\bbT$ be an arc of length $\abs{\delta}=2\pi/L$, $L\in\bbN$, 
and let $\delta_\ell$ be as in \eqref{c1}. Let
\begin{equation}
g=\sum_{\ell=1}^L c_\ell \1_{\delta_\ell}
\label{c3}
\end{equation}
for some coefficients $c_1,\dots,c_\ell\in\bbC$. Then
\begin{equation}
\Delta_{1/\gamma}(T_\gamma(g))=
\delta_{1/\gamma}(T_\gamma(g))=
\varkappa_\gamma
\int_{0}^{2\pi} \abs{g(e^{i\theta})}^{1/\gamma} \frac{d\theta}{2\pi}.
\label{c2}
\end{equation}
\end{lemma}
\begin{proof}
We have
$$
T_\gamma(g)=\sum_{\ell=1}^L c_\ell T_\gamma(\1_{\delta_\ell}), 
$$
and 
$$
T_\gamma(\1_{\delta_\ell})T_\gamma(\1_{\delta_j})\in\Sch_{1/2\gamma,\infty}^0, 
\quad
j\not=\ell.
$$
By Theorem~\ref{thm.b2}, we get
\begin{multline*}
\Delta_{1/\gamma}(T_\gamma(g))
\leq
\sum_{\ell=1}^L \Delta_{1/\gamma}(c_\ell T_\gamma(\1_{\delta_\ell}))
=
\sum_{\ell=1}^L \abs{c_\ell}^{1/\gamma} \Delta_{1/\gamma}(T_\gamma(\1_{\delta_\ell}))
\\
=
\frac1L \sum_{\ell=1}^L \abs{c_\ell}^{1/\gamma} \Delta_{1/\gamma}(T_\gamma(1))
=
\varkappa_\gamma 
\int_{0}^{2\pi} \abs{g(e^{i\theta})}^{1/\gamma} \frac{d\theta}{2\pi}
\end{multline*}
and similarly 
$$
\delta_{1/\gamma}(T_\gamma(g))
\geq
\frac1L \sum_{\ell=1}^L \abs{c_\ell}^{1/\gamma} \delta_{1/\gamma}(T_\gamma(1))
=
 \varkappa_\gamma 
\int_{0}^{2\pi} \abs{g(e^{i\theta})}^{1/\gamma} \frac{d\theta}{2\pi}.
$$
\end{proof}

%%%%%%%%%%%%%%%%%%%%%%%%%%%%%%%%%
\subsection{Concluding the proof}
%%%%%%%%%%%%%%%%%%%%%%%%%%%%%%%%%

%%%%%%%%%%%%%%%%
\begin{lemma}\label{lma.c7}
%%%%%%%%%%%%%%%%
Let $g\in C(\bbT)$. 
Then formula \eqref{c2} holds true. 
\end{lemma}
\begin{proof}
For any $\eps>0$, there exists a step function $g_\eps$ of the form \eqref{c3}
such that $\norm{g-g_\eps}_\infty\leq \eps$. 
By Lemma~\ref{lma.c6}, the identity
\begin{equation}
\Delta_{1/\gamma}(T_\gamma(g_\eps))=
\delta_{1/\gamma}(T_\gamma(g_\eps))=
\varkappa_\gamma
\int_{-\pi}^\pi \abs{g_\eps(e^{i\theta})}^{1/\gamma} \frac{d\theta}{2\pi}
\label{c2a}
\end{equation}
holds true for all $\eps>0$; our task is to pass to the limit as $\eps\to0$. 
It is obvious that one can pass to the limit in the right side of \eqref{c2a}. 
Consider the left side; by  Lemma~\ref{lma.c1a} we have
$$
%\Delta_{1/\gamma}(T_\gamma(g)-T_\gamma(g_\eps))=
\Delta_{1/\gamma}(T_\gamma(g-g_\eps))
\leq
2\varkappa_\gamma\eps^{1/\gamma}.
$$
Applying the estimate \eqref{b2}, we get
$$
\abs{
\Delta_{1/\gamma}(T_\gamma(g))^{\gamma/(\gamma+1)}
-
\Delta_{1/\gamma}(T_\gamma(g_\eps))^{\gamma/(\gamma+1)}
}
\leq
\Delta_{1/\gamma}(T_\gamma(g-g_\eps))^{\gamma/(\gamma+1)}
\leq 
C_\gamma
\eps^{1/(\gamma+1)}.
$$
It follows that
$$
\lim_{\eps\to0}\Delta_{1/\gamma}(T_\gamma(g_\eps))=\Delta_{1/\gamma}(T_\gamma(g)).
$$
Similarly, using \eqref{b3} instead of \eqref{b2}, we obtain
$$
\lim_{\eps\to0}\delta_{1/\gamma}(T_\gamma(g_\eps))=\delta_{1/\gamma}(T_\gamma(g)).
$$
Now we can pass to the limit $\eps\to0$ in \eqref{c2a}, which gives the desired result. 
\end{proof}

\begin{proof}[Proof of Theorem~\ref{thm.a1}]
Write $\f=\f_0+\f_1$, where
$$
\f_0(re^{i\theta})=(1-r)^\gamma g(e^{i\theta}),  
$$
and 
$$
\f_1(z)=o((1-\abs{z})^\gamma), \quad \abs{z}\to1.
$$
By the previous step, we have that $T(\f_0)$ satisfies the required 
asymptotics. It remains to prove that $T(\f_1)\in \Sch_{1/\gamma,\infty}^0$. 
In order to do this, for any $\eps>0$ write 
$\varphi_1=\psi_\eps+\wt \psi_\eps$, where 
$\wt \psi_\eps$ is supported inside the smaller disk $\abs{z}<a$, $a<1$, 
and $\psi_\eps$ satisfies the estimate
$$
\abs{\psi_\eps(z)}\leq \eps(1-\abs{z})^\gamma, \quad \abs{z}<1.
$$
It is easy to see that $T(\wt \psi_\eps)\in\Sch_0$. 
On the other hand, by  Lemma~\ref{lma.c1a}, we have
$$
\Delta_{1/\gamma}(T(\psi_\eps))
\leq
2\varkappa_\gamma \eps^{1/\gamma}. 
$$
By Lemma~\ref{lma.b1}, we get
$$
\Delta_{1/\gamma}(T(\f_1))
=
\Delta_{1/\gamma}(T(\psi_\eps))\leq 2\varkappa_\gamma \eps^{1/\gamma}.
$$
Since $\eps$ is arbitrary, we get $\Delta_{1/\gamma}(T(\f_1))=0$, which means
$T(\f_1)\in \Sch_{1/\gamma,\infty}^0$. 
\end{proof}

\section{Proof of Theorem~\ref{thm.a3}}\label{sec.e}
%%%%%%%%%%%%%%%%%%%%%%%%%%%%%%%%%%%%%%%%%%
%%%%%%%%%%%%%%%%%%%%%%%%%%%%%%%%%%%%%%%%%%

Let $\f(e^{i\theta})=(1-\abs{z})^\gamma b(e^{i\theta})$ with $b$ as in \eqref{e4}; consider the corresponding
Toeplitz operator $T(\f)$ in $B^2(\bbD)$. 
Let $T=\{t_{n,m}\}_{n,m=0}^\infty$ be the matrix of $T(\f)$ in the orthogonal basis 
$\{\sqrt{k+1}z^k\}_{k=0}^\infty$:
$$
t_{j,k}=\sqrt{j+1}\sqrt{k+1}(T(\f)z^j,z^k).
$$
We have
$$
t_{j,j+m}=0 \quad\text{ if $\abs{m}>M$.}
$$
Further, for $\abs{m}\leq M$ we have
$$
t_{j,j+m}
=
b_m\sqrt{j+1}\sqrt{j+m+1}((1-\abs{z})^\gamma e^{im\theta}z^j,z^{j+m})
=
2^{-\gamma}\Gamma(\gamma+1)b_m j^{-\gamma}+o(j^{-\gamma})
$$
as $j\to\infty$. 
This calculation shows that 
$$
2^\gamma(\Gamma(\gamma+1))^{-1}T
=
A+A', 
$$
where $A'$ is a banded matrix with 
$$
a'_{j,j+m}=o(j^{-\gamma}), \quad j\to\infty
$$
for all $\abs{m}\leq M$. 
Considering $A'$ as a sum of $2M+1$ matrices, each of which has non-zero 
entries only on the ``off-diagonal" $k=j+m$, it is easy to see that 
$A'\in\Sch_{1/\gamma,\infty}^0$. 
Thus, by Lemma~\ref{lma.b1}, 
$$
\Delta_{1/\gamma}(A)
=
\Delta_{1/\gamma}(2^\gamma(\Gamma(\gamma+1))^{-1}T)
=
2(\Gamma(\gamma+1))^{-1/\gamma}\Delta_{1/\gamma}(T).
$$
Finally, by Theorem~\ref{thm.a1},
$$
\Delta_{1/\gamma}(A)
=
2(\Gamma(\gamma+1))^{-1/\gamma}\Delta_{1/\gamma}(T)
=
\int_{-\pi}^\pi \abs{b(e^{i\theta})}^{1/\gamma}\frac{d\theta}{2\pi}.
$$
The same calculation applies to $\delta_{1/\gamma}(A)$. 
This completes the proof of Theorem~\ref{thm.a3}.


\begin{thebibliography}{7}


\bibitem{BS}
{\sc M.~Sh.~Birman, M.~Z.~Solomyak,}
\emph{Spectral theory of self-adjoint operators in Hilbert space,}
Reidel, 1987.


\bibitem{BS2}
{\sc M.~Sh.~Birman, M.~Z.~Solomyak, } 
\emph{Compact operators with power asymptotic behavior of the singular numbers,} 
J. Sov. Math. \textbf{27} (1984), 2442--2447. 



\bibitem{BR}
{\sc V.~Bruneau, G.~Raikov,}
\emph{Spectral Properties of Harmonic Toeplitz Operators and Applications to the Perturbed Krein Laplacian,}
preprint, arXiv:1609.08229. 

\bibitem{PuYa}
{\sc A.~Pushnitski, D.~Yafaev,}
\emph{Localization principle for compact Hankel operators,}
J. Funct. Anal. \textbf{270} (2016), 3591--3621. 

\bibitem{PuYa2}
{\sc A.~Pushnitski, D.~Yafaev,}
\emph{Spectral asymptotics for compact self-adjoint Hankel operators,}
J. Operator Theory \textbf{74} no.2 (2015), 417--455. 


\bibitem{Raikov}
{\sc G.~D.~Raikov,}
\emph{Eigenvalue asymptotics for the Schr\"odinger operator,} 
Communications in Partial Differential Equations, \textbf{15} no.3 (1990), 407--434.




\end{thebibliography}
\end{document}